\documentclass[a4paper,11pt,reqno]{amsart}

\usepackage{amsfonts,amsmath,amssymb,amsthm}
\usepackage{graphicx}
\usepackage{fixmath}
\usepackage{hyperref}
\usepackage[english]{babel}

\usepackage{setspace}
\usepackage{subfig}
\usepackage{url}
\usepackage{booktabs}
\usepackage{ifthen}
\usepackage{wrapfig}

\usepackage{enumerate}

\usepackage{tikz}
\usetikzlibrary{calc}
\usetikzlibrary{decorations.pathreplacing}

\usepackage[T2A,T1]{fontenc}

\newtheorem{thm}{Theorem}[section]
\newtheorem{lem}[thm]{Lemma}

\newtheorem{cor}[thm]{Corollary}
\newtheorem{obs}[thm]{Observation}

\newtheorem{defn}[thm]{Definition}

\newtheorem{conj}[thm]{Conjecture}
\numberwithin{equation}{section}

\begin{document}
 
\thispagestyle{plain}
 
%\title{Vizing's Conjecture for Graphs at Least as Large as the Product of their Domination Numbers}

\title{Vizing's Conjecture for Almost All Pairs of Graphs}

\author{Aziz Contractor}
\address{Aziz Contractor (\tt acontractor@student.clayton.edu)}

\author{Elliot Krop}
\address{Elliot Krop (\tt elliotkrop@clayton.edu)}
\address{Department of Mathematics, Clayton State University}
\date{\today}
\date{\today}

\begin {abstract}
For any graph $G=(V,E)$, a subset $S\subseteq V$ \emph{dominates} $G$ if all vertices are contained in the closed neighborhood of $S$, that is $N[S]=V$. The minimum cardinality over all such $S$ is called the domination number, written $\gamma(G)$. In 1963, V.G. Vizing conjectured that $\gamma(G \square H) \geq \gamma(G)\gamma(H)$ where $\square$ stands for the Cartesian product of graphs. In this note, we prove that if $\left|G\right|\geq \gamma(G)\gamma(H)$ and $\left|H\right|\geq \gamma(G)\gamma(H)$, then the conjecture holds. This result quickly implies Vizing's conjecture for almost all pairs of graphs $G,H$ with $\left|G\right|\geq \left|H\right|$, satisfying $\left|G\right|\leq q^{\frac{\left|H\right|}{\log_q\left|H\right|}}$ for $q=\frac{1}{1-p}$ and $p$ the edge probability of the Erd\H{o}s-R\'enyi random graph.
\\[\baselineskip] 2010 Mathematics Subject
      Classification: 05C69
\\[\baselineskip]
\end {abstract}

\maketitle

 \section{Introduction}
 For basic graph theoretic notation and definitions see Diestel~\cite{Diest}. All graphs $G(V,E)$ are finite, simple, undirected graphs with vertex set $V$ and edge set $E$. We may refer to the vertex set and edge set of $G$ as $V(G)$ and $E(G)$, respectively.
 
 For any graph $G=(V,E)$, a subset $S\subseteq V$ \emph{dominates} $G$ if $N[S]=G$. The minimum cardinality of $S \subseteq V$, so that $S$ dominates $G$ is called the \emph{domination number} of $G$ and is denoted $\gamma(G)$.
 
 \begin{defn}
The \emph{Cartesian product} of two graphs $G_1(V_1,E_1)$ and \linebreak$G_2(V_2,E_2)$, denoted by $G_1 \square G_2$, is a graph with vertex set $V_1 \times V_2$ and edge set $E(G_1 \square G_2) = \{((u_1,v_1),(u_2,v_2)) : v_1=v_2 \mbox{ and } (u_1,u_2) \in E_1, \mbox{ or } u_1 = u_2 \mbox{ and } (v_1,v_2) \in E_2\}$.
\end{defn}

The famous conjecture of Vadim G. Vizing (1963) \cite{Vizing} states that
\begin{align}
\gamma(G \square H) \geq \gamma(G)\gamma(H).\label{V}
\end{align}

Previous work on this problem has been reviewed in the excellent survey \cite{BDGHHKR}.

One of the earliest significant results is that of Barcalkin and German \cite{BG}, who showed that the conjecture holds for decomposable graphs, that is, graphs $G$ with vertex sets which can be disjointly covered by $\gamma(G)$ cliques. 

A generalization of those techniques came much later in \cite{BR}. The authors defined the related parameter of fair domination and showed that graphs with identical fair domination number and domination number satisfy the conjecture. However, finding bounds on fair domination numbers has been diffcult so far.

The best current bound for the conjectured inequality was shown in 2010 by Suen and Tarr \cite{ST}, 
\[\gamma(G \square H) \geq \frac{1}{2}\gamma(G)\gamma(H)+\frac{1}{2}\min\{\gamma(G),\gamma(H)\}.\]

We take a different point of view and show that for any two fixed domination numbers, $\gamma_1$ and $\gamma_2$, all graphs attaining those domination numbers, respectively, with orders larger than the product $\gamma_1 \gamma_2$, satisfy the conjecture. The proof of this fact is an elementary counting argument. By applying a result of Dryer \cite{Dryer} it is easy to show that Vizing's conjecture holds for almost all graphs $G,H$ with $\left|G\right|\geq \left|H\right|$ satisfying the order bound condition $\left|G\right|\leq q^{\frac{\left|H\right|}{\log_q\left|H\right|}}$ for $q=\frac{1}{1-p}$ and $p$ the edge probability of the Erd\H{o}s-R\'enyi random graph.

\section{Counting Vertices in Blocks}
%For any two graphs $G$ and $H$, and any vertex $h\in V(H)$, the $G$-fiber, $G^h$, is the subgraph of $G\square H$ induced by $\{(g,h):g\in V(G)\}$. For any $h\in H$ and $i$, $1\leq i \leq k$, call $S_i^h=S_i\square \{h\}$ a \emph{cell}. 

Given vertex partitions of $G$ into sets $G_1, \dots, G_k$ and $H$ into sets $H_1, \dots, H_l$, a \emph{block} of $G \square H$ is the induced subgraph $G_i \square H_j$, for some $i$, $1\leq i \leq k$, and $j$, $1\leq j \leq l$.

\begin{thm}
For every two graphs $G$ and $H$ satisfying $\left|G\right|\geq \gamma(G)\gamma(H)$ and $\left|H\right|\geq \gamma(G)\gamma(H)$, $\gamma(G \square H) \geq \gamma(G)\gamma(H)$. 
\end{thm}

\begin{proof}
Let $D$ be a minimum dominating set of $G \square H$. Suppose $\gamma(G)=k$ and $\gamma(H)=l$. Partition the vertices of $G$ arbitrarily into sets $G_1,\dots, G_k$ and the vertices of $H$ into sets $H_1, \dots, H_l$ so that for any $i$, $1\leq i \leq k$, and any $j$, $1\leq j \leq l$, $\left|G_i\right|\geq \gamma(H)$ and $\left|H_j\right|\geq \gamma(G)$. Furthermore, we call a block $B_{i,j}=G_i \square H_j$ a \emph{G-cell block} if there are at least $\left|H_j\right|$ vertices of $D$ in $G \square H_j$. We say $B$ is a \emph{H-cell block} if there are at least $\left|G_i\right|$ vertices of $D$ in $G_i \square H$.

\begin{obs}
Every block is either a $G$-cell block or an $H$-cell block.
\end{obs}

%Our goal is to show that after re-partitioning the vertices of $G$ and $H$ into blocks that retain the above size condition, we can find at least $\max\{\gamma(G), \gamma(H)\}$ blocks containing at least $\min\{\gamma(G), \gamma(H)\}$ vertices of $D$.

Without loss of generality, suppose $\gamma(G)\geq \gamma(H)$. If no block \linebreak $\{B_{1,1}, B_{2,1}, \dots, B_{k,1}\}$ is a $G$-cell block, then each is an $H$-cell block and we count at least $\gamma(G)\gamma(H)$ vertices of $D$. Thus, we can find at least one block in the above list which is a $G$-cell block, and by definintion, all the blocks in the list are $G$-cell blocks. Call the vertices of $G_1$, $\{v_1, v_2, \dots, v_l\}$. Define \[P_{i,j}=\{u \in G_i: (u,v)\in D \mbox{ for some }v \in H_j\}.\]
That is, $P_{i,j}$ is the projection of the vertices of $D$ in block $B_{i,j}$ onto $G$.

We call the following procedure the \emph{re-partitioning argument}, which we apply for part $G_1$ of the partition.

Notice that for any $v\in G_1$, if $v\notin P_{1,1}$, since $B_{1,1}$ is a $G$-cell block, there exists some vertex $u\in P_{i,1}$, for $i\geq 2$. Furthermore, such vertices $u$ can be chosen distinctly for every $v$, and so we can define an injective function $f_{B_{1,1}}:\{v\in G_1: v\notin P_{1,1}\}\rightarrow V(G) \backslash G_1$  so that $f_{B_{1,1}}(v)=u$ for $v$ and $u$ as defined above. 

We re-partition $G$ by exchanging every vertex  $v\in G_1$ such that $v\notin P_{1,1}$ with $f_{B_{1,1}}(v)$, and calling the new set of vertices $\hat{G}_1^1$. Call the remaining sets of the partition $\hat{G}_2^1, \dots, \hat{G}_k^1$. Next, for every $i$, $2\leq i \leq k$, we remove vertices of $G_1^1$ which are not in $P_{1,1}$ and add them arbitrarily to other parts, remove $P_{i,1}$ from $\hat{G}_i^1$ to form $G_i^1$, and append ${G}_1^1$ by $P_{i,1}$ to define the vertex partition $G_1^1=\hat{G}_1^1 \cup (\cup_{i=2}^k P_{i,1}), G_2^1, \dots, G_k^1$. We call the blocks of this partition $B_{i,j}^1$ for $1\leq i \leq k$ and $1\leq j \leq l$.

We note that 

\begin{enumerate}
\item The new block $B_{1,1}^1$ is a $G$-cell block and contains at least $\gamma(G)$ vertices of $D$.
\item The new blocks $B_{i,1}^1$, for $2\leq i \leq k$, contain no vertices of $D$.
\item Some or all of the new blocks $B_{i,1}^1$, for $2\leq i \leq k$, may be empty.
\end{enumerate}

If all blocks $B_{i,1}^1$ are empty for $2\leq i \leq k$, then $G_1^1$ contains all the vertices of $G$ and for every vertex $v\in G$, there is a vertex of $D$ in $\{v\}\square H$. Since $\left|G\right|\geq \gamma(G)\gamma(H)$, the conjecture holds.

Next, if no block $\{B_{2,2}^1, B_{3,2}^1, \dots, B_{k,2}^1\}$ is a $G$-cell block, then each is an $H$-cell block and each block $B_{i,2}^1$ contains at least $\left|G_i^1\right|$ vertices of $D$, for $2\leq i \leq k$. Since the vertices of $D\cap B_{1,1}^1$ do not appear among these, and $B_{1,1}^1$ contains at least $\left|G_1^1\right|$ vertices of $D$, we count at least $\left|G\right|\geq \gamma(G)\gamma(H)$ vertices of $D$. This leaves us with the case when $B_{2,2}^1$ is a $G$-cell block.

We repeat the previous re-partitioning argument for the part $G_2^1$ without altering $G_1^1$. Define an injective function $f_{B_{2,2}}:\{v\in V(G_2^1): v\notin P_{2,2}\}\rightarrow V(G) \backslash (G_1^1\cup G_2^1)$  so that $f_{B_{2,2}}(v)=u$ for $v\in G_2^1$, $v\notin P_{2,2}$, and $u\in P_{i,2}$, for $i\geq 3$.

We exchange every vertex $v\in G_2^1$ such that $v\notin P_{2,2}$, with $f_{B_{2,2}^1}(v)$, and call the new set of vertices $\hat{G}_2^2$. Call the remaining new sets of the partition $\hat{G}_3^2, \dots, \hat{G}_k^2$. Next, for every $i$, $3\leq i \leq k$, we remove vertices of $G_2^2$ which are not in $P_{2,2}$ and add them arbitrarily to parts other than $G_1^2$, we remove $P_{i,2}$ from $\hat{G}_i^2$ to form $G_i^2$, and append ${G}_1^2$ by $P_{i,2}$ to define the vertex partition $G_1^2=\hat{G}_1^2 \cup (\cup_{i=3}^k P_{i,2}), G_3^2, \dots, G_k^2$. We call the blocks of this partition $B_{i,j}^2$ for $1\leq i \leq k$ and $1\leq j \leq l$.

We note that 

\begin{enumerate}
\item The block $B_{2,2}^2$ is a $G$-cell block and $B_{1,2}^2\cup B_{2,2}^2$ contain at least $\gamma(G)$ vertices of $D$.
\item The new blocks $B_{i,2}^2$, for $3\leq i \leq k$, contain no vertices of $D$.
\item Some or all of the new blocks $B_{i,2}^2$, for $3\leq i \leq k$, may be empty.
\end{enumerate}

If all blocks $B_{i,2}^2$ are empty for $3\leq i \leq k$, then $G_1^2\cup G_2^2$ contains all the vertices of $G$ and for every vertex $v\in G$, there is a vertex of $D$ in $\{v\}\square H$. Since $\left|G\right|\geq \gamma(G)\gamma(H)$, the conjecture holds.

Again, if no block $\{B_{3,3}^2, B_{4,3}^2, \dots, B_{k,3}^2\}$ is a $G$-cell block, then each is an $H$-cell block and each block $B_{i,3}^2$ contains at least $\left|G_i^2\right|$ vertices of $D$, for $3\leq i \leq k$. Since the vertices of $D\cap B_{1,1}^2$ and $D\cap (B_{1,2}^2\cup B_{2,2}^2)$ do not appear among these, and they contain at least $\left|G_1^2\right|$ and $\left|G_2^2\right|$ vertices of $D$ respectively, we count at least $\left|G\right|\geq \gamma(G)\gamma(H)$ vertices of $D$. This leaves us with the case when $B_{2,2}^1$ is a $G$-cell block.

We continue re-partitioning for every set $G_i^{i-1}$ for $3\leq i \leq l-1$ so that

\begin{enumerate}
\item The block $B_{i,i}^i$ is a $G$-cell block and $B_{1,i}^i\cup B_{2,i}^i\cup \dots \cup B_{i,i}^i$ contain at least $\gamma(G)$ vertices of $D$.
\item The new blocks $B_{j,i}^i$, for $i+1\leq j \leq k$, contain no vertices of $D$.
\item Some or all of the new blocks $B_{j,i}^i$, for $i+1\leq j \leq k$, may be empty.
\end{enumerate}

Suppose the re-partitioning algorithm terminates for some $i=m$. Summing the number of vertices of $D$ in the blocks 
\begin{align*}
&B_{1,1}^m\\
&B_{1,2}^m\cup B_{2,2}^m\\
&\vdots\\
&B_{1,m}^m\cup \dots \cup B_{m,m}^m
\end{align*}
produces at least $\gamma(G)\gamma(H)$ vertices.

\end{proof}

For the probabilistic result, we use the Erd\H{o}s-R\'{e}nyi random graph model \cite{B}, $G_{n,p}$, where a graph contains $n$ vertices and each pair of vertices is joined by an edge with probability $p$.  Dryer \cite{Dryer} showed

\begin{lem}\cite{Dryer}
Choose $p\in [0,1) $ and $T$ any vertex set of size $(1+\epsilon)\log_qn$ in $G_{n,p}$, where $\epsilon>0$ and $q=\frac{1}{1-p}$. Then $Pr(T \mbox{ is a dominating set})$ approaches $1$ as $n$ approaches infinity.
\end{lem}

Applying Dryer's result to the condition $\left|G\right|\geq \gamma(G)\gamma(H)$ and $\left|H\right|\geq \gamma(G)\gamma(H)$ produces

\begin{cor}
Vizing's conjecture holds for almost all pairs of graphs $G,H$ with $\left|G\right|\geq \left|H\right|$, satisfying $\left|G\right|\leq q^{\frac{\left|H\right|}{\log_q\left|H\right|}}$ for $q=\frac{1}{1-p}$ and $p$ the edge probability of the Erd\H{o}s-R\'enyi random graph.
\end{cor}

It would be interesting to prove the following

\begin{conj}
Vizing's conjecture holds for almost all pairs of graphs.
\end{conj}

 \bibliographystyle{plain}
 
 \end{document}